\documentclass[12pt]{amsart}

\usepackage{ucs}

\usepackage{amssymb}
\usepackage{amsthm}
\usepackage{amsmath}
\usepackage{latexsym}
\usepackage[cp1251]{inputenc}
\usepackage{graphicx}
\usepackage{wrapfig}
\usepackage{caption}
\usepackage{subcaption}
\usepackage{indentfirst}
\usepackage[left=2.6cm,right=2.6cm,top=2.6cm,bottom=2.6cm,bindingoffset=0cm]{geometry}
\usepackage{enumerate}
\usepackage{makecell}
\usepackage{float}

\DeclareMathOperator{\cay}{Cay}

\def\tm#1{\item[{\rm (#1)}]}

\makeatletter 
\def\@seccntformat#1{\csname the#1\endcsname. } 
\def\@biblabel#1{#1.} 

\newcommand{\overbar}[1]{\mkern 1.5mu\overline{\mkern-1.5mu#1\mkern-1.5mu}\mkern 1.5mu}

\makeatother

\title{Deza Cayley graphs from difference sets}

\author{Grigory Ryabov}

\address{Sobolev Institute of Mathematics, Novosibirsk, Russia}

\email{gric2ryabov@gmail.com}

\thanks{The author was supported by the state contract of the Sobolev Institute of Mathematics (project number FWNF-2022-0017)}

\date{}

\newtheorem{prop}{Proposition}[section]

\newtheorem{lemm}[prop]{Lemma}
\newtheorem{theo}[prop]{Theorem}

\theoremstyle{definition}
\newtheorem{defn}{Definition}[section]

\begin{document}

\maketitle

\begin{abstract}
In this note, we provide several constructions of Deza Cayley graphs over groups having a generalized dihedral subgroup. These constructions are based on usage of (relative) difference sets. 
\\
\\
\textbf{Keywords}: Deza graphs, Cayley graphs, difference sets.

\noindent\textbf{MSC}: 05B10, 20C05, 05E30. 
\end{abstract}

\maketitle

\section{Introduction}

A regular graph\footnote{Throughout the paper, by a graph we mean a simple undirected graph without loops.} $\Gamma$ is said to be \emph{strongly regular} if there exist nonnegative integers $\lambda$ and $\mu$ such that every two adjacent vertices of $\Gamma$ have exactly $\lambda$ common neighbors, whereas every two distinct nonadjacent vertices of $\Gamma$ have exactly $\mu$ common neighbors. Nowadays, strongly regular graphs are realized as one of the keynote objects in algebraic combinatorics. For a background of strongly regular graphs, we refer the readers to the monograph~\cite{BM}.

The following generalization of the notion of strongly regular graph going back to~\cite{Deza} was suggested in~\cite{EFHHH}. 

\begin{defn}
A regular graph $\Gamma$ is called a \emph{Deza} graph if there exist nonnegative integers $a$ and $b$ such that any two distinct vertices of $\Gamma$ have either $a$ or $b$ common neighbors. 
\end{defn}

\noindent The numbers $(v,k,b,a)$, where $v$ is the number of vertices of $\Gamma$ and $k$ is the degree of each vertex, are called the \emph{parameters} of $\Gamma$. We may always assume that $a\leq b$. Clearly, if $a>0$ and $b>0$, then $\Gamma$ has diameter~$2$. A Deza graph is called a \emph{strictly Deza graph} if it is not strongly regular and has diameter~$2$.

Deza graphs attract attention of several mathematicians. This class of graphs has been actively studied during the last years. Deza graphs with special parameters were investigated in~\cite{GHKS,KMS,KS}. Some results on spectra of Deza graphs were obtained in~\cite{AGHKKS,AHHKKS}. A vertex connectivity of Deza graphs was studied in~\cite{GGK,GP}. The class of Deza graphs with strongly regular children was a research object in~\cite{KKS}. For more details on Deza graphs and a recent progress in their studying, we refer the readers to the survey paper~\cite{GSh2}. Computational results on Deza graphs can be found in~\cite{GPSh}.

In the present paper, we are interested in Deza Cayley graphs. By a \emph{Cayley graph} $\Gamma=\cay(G,S)$ over a finite group $G$ with a non-empty inverse-closed identity-free connection set $S\subseteq G$, we mean the graph with vertex set $G$ and edge set $E=\{(g,xg):~x\in X,~g\in G\}$. All Deza Cayley graphs with at most~$60$ vertices were enumerated in~\cite{GSh}. Several results on Deza Cayley graphs over cyclic groups are given in~\cite{BPR} and~\cite[Section~2]{GSh2}, whereas several results on Deza Cayley graphs over dihedral groups are given in~\cite{RS}. A construction of a Deza Cayley graph with the largest possible association scheme and the smallest possible automorphism group can be found in~\cite{CR}.

One of the main problems concerned with Deza graphs is the problem of constructing new families of them. In this paper, we provide several constructions of Deza Cayley graphs. Recall that a generalized dihedral group associated with an abelian group~$A$ is defined to be a semidirect product of~$A$ and a cyclic group of order~$2$ whose nontrivial element inverses every element of~$A$. The main result of the paper is the theorem below stating an existence of two new infinite families of strictly Deza Cayley graphs over generalized dihedral groups.

\begin{theo}\label{main1}
There exist strictly Deza Cayley graphs with parameters
$$(2^{2k},2^{2k}-2^k-1,2^{2k}-2^{k+1},2^{2k}-2^{k+1}-2)$$
and
$$(q^2-1,q^2-q-2,q^2-2q-1,q^2-2q-3),$$
over the generalized dihedral groups associated with the groups~$\mathbb{Z}_4^{k-1}\times \mathbb{Z}_2$ and $\mathbb{Z}_{\frac{q^2-1}{2}}$, respectively, where $k\geq 3$ and $q\geq 5$ is an odd prime power.
\end{theo}

As we can check using the database of results on Deza graphs~\cite{Pan}, the graphs from Theorem~\ref{main1} are previously unknown except for the graphs with parameters~$(24,18,14,12)$ and~$(48,40,34,32)$ from the second family for $q=5$ and $q=7$ which appear among small Deza Cayley graphs from~\cite{GSh}. The constructions of Deza Cayley graphs from Theorem~\ref{main1} are based on usage of relative difference sets. Our technique enables us to find some more sporadic Deza Cayley graphs. 

\begin{theo}\label{main2}
There exists a Deza Cayley graph with parameters~$(v,k,b,a)$ from the first column of Table~$1$ over a generalized dihedral group associated with an abelian group~$A$ from the second column of Table~$1$.
\end{theo}

\begin{table}[h]
\centering
{\small
\begin{tabular}{|l|l|l|}
  \hline
  $(v,k,b,a)$ & $A$ & strictly  \\
  \hline
 $(14,9,6,4)$    & $\mathbb{Z}_7$  &  yes   \\ \hline
 $(16,5,2,0)$    &  $\mathbb{Z}_4\times \mathbb{Z}_2$, $\mathbb{Z}_2^3$ & no  \\  \hline
 $(16,11,8,6)$  &  $\mathbb{Z}_4\times \mathbb{Z}_2$, $\mathbb{Z}_2^3$  & yes\\  \hline
 $(22,16,12,10)$ & $\mathbb{Z}_{11}$  &  yes \\  \hline
 $(24,6,2,0)$ &  $\mathbb{Z}_{12}$  & no  \\ \hline
 $(32,25,20,18)$ & $\mathbb{Z}_8\times \mathbb{Z}_2$, $\mathbb{Z}_4^2$, $\mathbb{Z}_4\times \mathbb{Z}_2^2$, $\mathbb{Z}_2^4$  & yes \\  \hline
 $(64,11,2,0)$   & $\mathbb{Z}_4^2\times \mathbb{Z}_2$ &  yes\\  \hline
 $(74,64,56,54)$ & $\mathbb{Z}_{37}$  & yes\\  \hline
 $(80,12,2,0)$  &  $\mathbb{Z}_{40}$ & no \\  \hline
 $(252,235,220,218)$  & $\mathbb{Z}_{126}$  & yes\\  \hline
\end{tabular}
}
\caption{Sporadic Deza Cayley graphs}
\end{table}
  
As we are able to verify, the last four graphs from Table~$1$ are previously unknown, whereas the other ones can be found in the list of small Deza Cayley graphs from~\cite{GSh}. The graph with parameters~$(16,5,2,0)$ is isomorphic to the strongly regular van Lint-Schrijver graph (see~\cite[p.~371]{BM}).

Let us finish the introduction with a brief outline of the paper. Deza graphs from Theorems~\ref{main1} and~\ref{main2} are constructed from difference sets and relative difference sets (see Propositions~\ref{ds} and~\ref{rds}). A necessary information and families of (relative) difference sets used for constructing Deza Cayley graphs are given in Section~$2$. One of the main tools used in the proof of Proposition~\ref{rds} is computations in the integer group ring of $G$. A short background of group rings is given in Section~$3$. In Section~$4$, we provide constructions of Deza Cayley graphs and prove Theorems~\ref{main1} and~\ref{main2}. We observe in Section~$5$ that parameters of the graphs from Theorems~\ref{main1} and~\ref{main2} satisfy some arithmetic conditions which allows to construct from these graphs more Deza graphs using some operations on graphs. Several concluding remarks are given in Section~$6$.
 
\hspace{5mm}

\noindent \textbf{Acknowledgement:} The author would like to thank the anonymous referee for comments which help to improve the text substantially.

\section{Difference sets}

In this section, we give necessary information on difference sets. At first, let us recall the definition of a relative difference set. 

\begin{defn}
A subset $R$ of $G$ is called a \emph{relative difference set} (\emph{RDS} for short) in $G$ if there exist a subgroup~$N$ of~$G$ and a nonnegative integer $\lambda$ such that every element of $G\setminus N$ has exactly $\lambda$ representations in the form $g_1g_2^{-1}$, where $g_1,g_2\in R$, and no non-identity element of~$N$ has such a representation. 
\end{defn}

\noindent The subgroup~$N$ and the numbers $(m,n,k,\lambda)$, where $m=|G:N|$, $n=|N|$, and $k=|R|$, are called the \emph{forbidden subgroup} and \emph{parameters} of~$R$, respectively. The RDS $R$ is said to be \emph{nontrivial} if $2\leq |R|\leq |G|-2$. The latter implies that $\lambda>0$ and $m>1$. It is easy to check that $R$ contains at most one element from each right $N$-coset. If $R$ contains exactly one element from each right $N$-coset, i.e. $R$ is a right transversal for $N$ in $G$, then $R$ is said to be \emph{semiregular}. In this case, $m=k=\lambda n$. 

If $N$ is trivial, then $R$ is called a \emph{difference set} (\emph{DS} for short). If $R$ is a DS, then the parameters of $R$ are written as~$(v,k,\lambda)$, where $v=|G|$. It is well-known that if $R$ is a DS in~$A$ with parameters~$(v,k,\lambda)$, then $(A\setminus R)$ is a DS in~$A$ with parameters 
$$(v,v-k,v-2k+\lambda).$$ 

An RDS is a special case of a divisible difference set which can be considered as a Cayley object, i.e. a combinatorial object having a regular subgroup in its automorphism group, in the class of divisible designs. For a background of DSs and RDSs, we refer the readers to~\cite[Section~VI]{BJL} and~\cite{Pott2}, respectively.

Below, we provide families of DSs and RDSs which will be used for constructing Deza Cayley graphs. 

\begin{lemm}\label{smallds}
There exist DSs with parameters~$(7,3,1)$, $(11,6,3)$, $(16,10,6)$, and $(37,28,21)$ in abelian groups.
\end{lemm}

\begin{proof}
According to~\cite[Table~A.3.1]{BJL}, there exist DSs with parameters~$(7,3,1)$, $(11,5,2)$, $(16,6,2)$, and $(37,9,2)$ in abelian groups. The first of them and the complements to the other ones are the required DSs in the lemma. 
\end{proof}

\begin{lemm}\label{smallrds}
The following statements hold.
\begin{enumerate}

\tm{1} If $k\geq 3$, then the group $\mathbb{Z}_4^{k-1}\times \mathbb{Z}_2$ has an RDS with parameters~$(2^k,2^{k-1},2^k,2)$.

\tm{2} If $q\geq 5$ is an odd prime power, then the group $\mathbb{Z}_{\frac{q^2-1}{2}}$ has an RDS with parameters $(q+1,\frac{q-1}{2},q,2)$.

\tm{3} The groups $\mathbb{Z}_4\times \mathbb{Z}_2$ and $\mathbb{Z}_2^3$ have RDSs with parameters~$(4,2,4,2)$.

\tm{4} The group $\mathbb{Z}_{126}$ has an RDS with parameters~$(21,6,16,2)$.

\end{enumerate}

\end{lemm}

\begin{proof}
Statement~$(1)$ follows from~\cite[Theorem~$1$]{DJ} applied to $a=k$ and $b=k-1$. Statements~$(2)$ and~$(3)$ are special cases of~\cite[Theorem~$1.2$]{ADLM} when $d=2$ and $n=\frac{q-1}{2}$ and~\cite[Theorem~2.1]{MS} when $p=2$ and $c=1$, respectively. Finally, Statement~$(4)$ is taken from~\cite{Lam} (see also~\cite[Result~$3$]{ADLM}).
\end{proof}

\section{Group ring}

Let $G$ be a finite group and $\mathbb{Z}G$ the integer group ring. The identity element and the set of all nonidentity elements of $G$ are denoted by~$e$ and~$G^\#$, respectively. A product in the group ring of two elements $x=\sum_{g\in G} x_g g,y=\sum_{g\in G} y_g g\in\mathbb{Z}G$ will be written as $x\cdot y$. If $X\subseteq G$, then the set $\{x^{-1}:x\in X\}$ and the element $\sum \limits_{x\in X} {x}$ of the group ring $\mathbb{Z}G$ are denoted by~$X^{-1}$ and~$\underline{X}$, respectively. 

It is easy to verify that
\begin{equation}\label{easy}
\underline{X}\cdot \underline{H}=\underline{H}\cdot \underline{X}=|X|\underline{H}
\end{equation}
for all $X\subseteq G$ and $H\leq G$ such that $X\subseteq H$. Using Eq.~\eqref{easy}, one can deduce that
\begin{equation}\label{group}
(\underline{H}^\#)^2=(\underline{H}-e)^2=(|H|-1)e+(|H|-2)\underline{H}^\#.
\end{equation}

Let $N$ be a subgroup of $G$, $m=|G:N|$, and $n=|N|$. It follows easily from the definition of RDS that a subset $R$ of $G$ is an RDS with forbidden subgroup~$N$ and parameters~$(m,n,k,\lambda)$, where $k=|R|$ and $\lambda\geq 0$, if and only if
\begin{equation}\label{defrds}
\underline{R}\cdot \underline{R}^{-1}=ke+\lambda(\underline{G}-\underline{R}).
\end{equation} 

\begin{lemm}\label{semireg}
Let $R$ be an RDS in a group $G$ with a forbidden subgroup $N$. Then $\underline{N}\cdot \underline{R}=G$ if and only if $R$ is semiregular.
\end{lemm}

\begin{proof}
Recall that $R$ contains at most one element from each right $N$-coset. So the equality $\underline{N}\cdot \underline{R}=G$ is equivalent to the fact that $R$ is a transversal for $N$ in $G$ which means that $R$ is semiregular. 
\end{proof}

The following lemma provides an easy criterion for a Cayley graph to be a Deza graph in terms of the integer group ring. It can be found, e.g, in~\cite[Lemma~$5.2$]{BPR}. 

\begin{lemm}\label{deza}
Let $G$ be a finite group of order~$v$, $S\subseteq G$ such that $e\notin S$, $S=S^{-1}$, and $|S|=k$, $a$ and $b$ nonnegative integers, and $\Gamma=\cay(G,S)$. The graph $\Gamma$ is a Deza graph with parameters $(v,k,b,a)$ if and only if 
$$\underline{S}^2=ke+a\underline{S_a}+b\underline{S_b},$$
where $S_a$ and $S_b$ are subsets of $G$ such that $S_a\cup S_b=G^\#$ and $S_a\cap S_b=\varnothing$. Moreover, $\Gamma$ is strongly regular if and only if $S_a=S$ or $S_b=S$.
\end{lemm}

\section{Constructions}

In this section, we provide the constructions of Deza Cayley graphs based on usage of DSs and RDSs. Throughout the section, $A$ is an abelian group and $G=A\rtimes \langle b\rangle$, where $|b|=2$ and $a^b=a^{-1}$ for every $a\in A$, i.e. $G$ is a generalized dihedral group associated with~$A$.

The following statement in case of a dihedral group was proved in~\cite[Lemma~$3.1$]{RS}. However, the proof in case of a generalized dihedral group is exactly the same.

\begin{prop}\label{ds}
Let $R$ be a nontrivial DS in $A$ with parameters $(v,k,\lambda)$. Then the graph $\Gamma(R)=\cay(G,Rb\cup A^\#)$ is a Deza graph if and only if 
$$k=\frac{2v-1-\sqrt{8v-7}}{2}.$$ 
In this case, $\Gamma(R)$ is a strictly Deza graph and has parameters $(2v,v-1+k,2k,2(k-1))$. 
\end{prop}

\begin{prop}\label{rds}
Let $R$ be a nontrivial RDS in $A$ with a nontrivial forbidden subgroup $A_0$ and parameters $(m,n,k,\lambda)$. 
\begin{enumerate}

\tm{1} The graph $\Gamma_1(R)=\cay(G,S_1)$, where $S_1=Rb\cup A_0^\#$, is a Deza graph if and only if 
$$(n,\lambda)=(2,2)~\text{or}~(n,\lambda)=(4,2).$$
In this case, $\Gamma_1(R)$ has parameters 
$$(2mn,k+n-1,2,0)$$  
and it is of diameter~$2$ if and only if $R$ is semiregular. If the latter is the case, then $\Gamma_1(R)$ is strongly regular if and only if $(n,\lambda)=(2,2)$.

\tm{2} The graph $\Gamma_2(R)=\cay(G,S_2)$, where $S_2=G^\#\setminus Rb$, is a Deza graph if and only if 
$$\lambda=2.$$
In this case, $\Gamma_2(R)$ is a strictly Deza graph and has parameters 
$$(2mn,2mn-k-1,2mn-2k,2mn-2k-2).$$

\end{enumerate}

\end{prop}

\begin{proof}
Observe that $\lambda>0$ because $R$ is nontrivial. Let us prove Statement~$(1)$. Clearly, $|G|=2|A|=2mn$ and $|Rb\cup A_0^\#|=|Rb|+|A_0|-1=k+n-1$ and hence $\Gamma_1(R)$ is a $(k+n-1)$-regular graph on~$2mn$ vertices. A straightforward computation in the group ring of $G$ using the equality $a^b=a^{-1}$ for every $a\in A$ and Eqs.~\eqref{group} and~\eqref{defrds} implies that
\begin{equation}\label{compute1}
\begin{split}
\underline{S_1}^2=(\underline{Rb}+\underline{A_0}^\#)^2=\underline{R}\cdot \underline{R}^{-1}+(|A_0|-1)e+(|A_0|-2)\underline{A_0}^\#+2\underline{A_0}^\#\cdot \underline{R}b=\\
=(k+n-1)e+(n-2)\underline{A_0}^\#+\lambda(\underline{A}-\underline{A_0})+2\underline{A_0}^\#\cdot \underline{R}b.
\end{split}
\end{equation}
Since $R$ contains at most one element from each $A_0$-coset, every element of $A$ enters the element $\underline{A_0}^\#\cdot \underline{R}$ with coefficient~$0$ or~$1$. Moreover, every element of $R$ enters $\underline{A_0}^\#\cdot \underline{R}$ with coefficient~$0$ because $e\notin A_0^\#$. From Lemma~\ref{deza} and Eq.~\eqref{compute1} it follows that $\Gamma_1(R)$ is a Deza graph if and only if 
$$|\{n-2,\lambda,0,2\}|\leq 2$$ 
or, equivalently, 
$$\{n-2,\lambda\}\subseteq \{0,2\}.$$
Since $\lambda>0$, we conclude that $\lambda=2$ and hence $(n,\lambda)\in\{(2,2),(4,2)\}$. If the latter holds, then $\Gamma_1(R)$ has the required parameters by Lemma~\ref{deza} and Eq.~\eqref{compute1}. Observe that $\Gamma_1(R)$ has diameter~$2$ if and only if every element of
$$G^\#\setminus S_1=(A\setminus A_0) \cup (A\setminus R)b$$
enters the element $\underline{S_1}^2$ with a nonzero coefficient. Due to Eq.~\eqref{compute1}, this is equivalent to 
$$A\setminus R\subseteq A_0^\#R.$$
The latter inclusion holds if and only if $R$ is semiregular by Lemma~\ref{semireg}. The remaining part of Statement~$(1)$ follows from Lemma~\ref{deza} and Eq.~\eqref{compute1}.

Now let us prove Statement~$(2)$. In this case, $|G^\#\setminus Rb|=|G|-1-|Rb|=2mn-k-1$. So $\Gamma_2(R)$ is a $(2mn-k-1)$-regular graph on~$2mn$ vertices. One can compute using the equality $a^b=a^{-1}$ for every $a\in A$ and Eqs.~\eqref{easy}, \eqref{group}, and~\eqref{defrds} that
\begin{equation}\label{compute2}
\begin{split}
\underline{S_2}^2=(\underline{G}^\#-\underline{Rb})^2=(|G|-1)e+(|G|-2)G^\#+\underline{R}\cdot \underline{R}^{-1}-2\underline{G}^\#\cdot \underline{R}b=\\
=(2mn-k-1)e+2(mn-k-1)\underline{A_0}^\#+(2(mn-k-1)+\lambda)(\underline{A}-\underline{A_0})+\\
+2(mn-k)\underline{R}b+2(mn-k-1)(\underline{A}-\underline{R})b.
\end{split}
\end{equation}
Due to Lemma~\ref{deza} and Eq.~\eqref{compute2}, the graph $\Gamma_2(R)$ is a Deza graph if and only if 
$$|\{2(mn-k-1),2(mn-k),2(mn-k-1)+\lambda\}|\leq 2,$$
or, equivalently, 
$$2(mn-k-1)+\lambda\in \{2(mn-k-1),2(mn-k)\}.$$
Since $\lambda>0$, we obtain $\lambda=2$. The remaining part of Statement~$(2)$ is a consequence of Lemma~\ref{deza} and Eq.~\eqref{compute2}. 
\end{proof}

\begin{proof}[Proof of Theorems~\ref{main1} and~\ref{main2}]
Theorem~\ref{main1} is a consequence of Statement~$(2)$ of Proposition~\ref{rds} applied to the RDSs from Statements~$(1)$ and~$(2)$ of Lemma~\ref{smallrds}. 

It is easy to verify that the parameters of each DS from Lemma~\ref{smallds} satisfy the condition $k=\frac{2v-1-\sqrt{8v-7}}{2}$ and hence applying Proposition~\ref{ds} to these DSs, we obtain the strictly Deza Cayley graphs with parameters $(14,9,6,4)$, $(22,16,12,10)$, $(32,25,20,18)$, and~$(74,64,56,54)$ from Theorem~\ref{main2}. Other graphs from Theorem~\ref{main2} can be obtained by applying Statement~$(1)$ of Proposition~\ref{rds} to the RDSs from Statement~$(1)$ with $k=3$, Statement~$(2)$ with $q=5$ and $q=9$, and Statement~$(3)$ of Lemma~\ref{smallrds} and applying of Statement~$(2)$ of Proposition~\ref{rds} to the RDSs from Statements~$(3)$ and~$(4)$ of Lemma~\ref{smallrds}.
\end{proof}

\section{More Deza graphs}

In this section, we construct more Deza Cayley graphs using some operations. The first of the operations is the lexicographic product of graphs (see, e.g.,~\cite[p.~398]{EFHHH}). Recall that the \emph{lexicographic product} $\Gamma_1[\Gamma_2]$ of graphs $\Gamma_1=(V_1,E_1)$ and $\Gamma_2=(V_2,E_2)$ is defined to be the graph with vertex set $V=V_1\times V_2$ and edge set $E$ defined as follows:
$$((v_1,v_2),(u_1,u_2))\in E~\text{if and only if}~(v_1,u_1)\in E_1~\text{or}~v_1=u_1~\text{and}~(v_2,u_2)\in E_2.$$

Recall that the \emph{Cartesian product} $\Gamma_1\times \Gamma_2$ of $\Gamma_1$ and $\Gamma_2$ is defined to be the graph with vertex set $V=V_1\times V_2$ and edge set $E$ defined as follows:
$$((v_1,v_2),(u_1,u_2))\in E~\text{if and only if}~v_1=u_1~\text{and}~(v_2,u_2)\in E_2~\text{or}~v_2=u_2~\text{and}~(v_1,u_1)\in E_1.$$
The clique on $m$ vertices is denoted by~$K_m$ and the complement to the graph $\Gamma$ is denoted by $\overbar{\Gamma}$. The second operation is defined as follows:
$$\Gamma\mapsto \overbar{\overbar{\Gamma}\times K_2}.$$
Using the definitions, it is easy to verify that $\overbar{\overbar{\Gamma}\times K_2}$ is isomorphic to a disjoint union of two copies of $\Gamma$ in which each vertex from the one copy is adjacent to each vertex from the other one except its own copy.

\begin{lemm}\label{operat}
Let $\Gamma$ be a strictly Deza graph with parameters $(v,k,2k-v+2,2k-v)$. Then $K_m[\Gamma]$, $m\geq 1$, and $\overbar{\overbar{\Gamma}\times K_2}$ are strictly Deza graphs with parameters 
$$(mv,k+(m-1)v,(m-2)v+2k+2,(m-2)v+2k)$$ and 
$$(2v,k+v-1,2k,2k-2),$$
respectively. 
\end{lemm}

\begin{proof}
The graph $K_m[\Gamma]$ is a Deza graph with the required parameters by~\cite[Proposition~2.3]{EFHHH} applied to the strongly regular graph $\Gamma_1=K_m$ and the Deza graph $\Gamma_2=\Gamma$. 

By the definition, the graph $\Delta=\overbar{\overbar{\Gamma}\times K_2}$ is a $(k+v-1)$-regular graph on~$2v$ vertices. Let $u$ and $v$ be distinct vertices of $\Delta$. If $u$ and $v$ belong to the same copy of $\Gamma$, then they have $2k-v+2$ or $2k-v$ common neighbors in this copy and $v-2$ common neighbors in the other copy. Therefore they have $2k$ or $2k-2$ common neighbors. Suppose that $u$ and $v$ belong to distinct copies of $\Gamma$. If the copy of $u$ is adjacent to $v$, then $u$ and $v$ have $2k-2$ common neighbors, whereas otherwise they have $2k$ common neighbors. Thus, $\Delta$ is a Deza graph with parameters $(2v,k+v-1,2k,2k-2)$. Since $\Gamma$ is a strictly Deza graph, $\Delta$ so is.  
\end{proof}

Note that if $\Gamma$ is a Cayley graph, then $K_m[\Gamma]$, $m\geq 1$, and $\overbar{\overbar{\Gamma}\times K_2}$ so are. Indeed, let $\Gamma=\cay(H,S)$ for some group $H$ and an identity-free inverse-closed subset $S$ of $H$. Then 
$$K_m[\Gamma]\cong \cay(G,S\cup (G\setminus H)),$$ 
where $G$ is any group such that $G\geq H$ and $|G:H|=m$, and 
$$\overbar{\overbar{\Gamma}\times K_2}\cong \cay(H\times C,S\cup H^\#c),$$ 
where $C\cong \mathbb{Z}_2$ and $c$ is the nontrivial element of~$C$.

Let $\mathcal{K}$ be the class consisting of all Deza graphs from Theorem~\ref{main1} and the Deza graphs from Theorem~\ref{main2} with a nonzero fourth parameter. It can be verified by a straightforward computation that parameters of every graph from $\mathcal{K}$ satisfy the condition of Lemma~\ref{operat}, i.e. they are of the form $(v,k,2k-v+2,2k-v)$. One can see that if $\Gamma$ has parameters of the form $(v,k,2k-v+2,2k-v)$, then $K_m[\Gamma]$, $m\geq 1$, and $\overbar{\overbar{\Gamma}\times K_2}$ have parameters of this form. Therefore the closure of $\mathcal{K}$ under the operations 
$$\Gamma\mapsto K_m[\Gamma],~m\geq 1,~\text{and}~\Gamma\mapsto \overbar{\overbar{\Gamma}\times K_2}$$ 
consists of strictly Deza Cayley graphs.

Let $\mathcal{M}$ be the class consisting of all graphs from Theorem~\ref{main2} with a zero fourth parameter. One can see that every graph from $\mathcal{M}$ has parameters of the form~$(v,k,2,0)$. Due to~\cite[Theorem~2.8(iii)-(iv)]{EFHHH}, the closure of $\mathcal{M}$ under the operations 
$$\Gamma \mapsto \overbar{K}_m \times \Gamma,~m\geq 1~\text{and}~(\Gamma_1,\Gamma_2)\mapsto \Gamma_1\times \Gamma_2$$ 
consists of Deza Cayley graphs. One can verify straightforwardly using the definitions of the above operations that all of these graphs are non-strictly Deza graphs except for the graph with parameters~$(64,11,2,0)$ from Theorem~\ref{main2}.

\section{Concluding remarks}

We do not know whether there exist other DSs with parameters satisfying the condition $k=\frac{2v-1-\sqrt{8v-7}}{2}$ which can be used for the construction from Proposition~\ref{ds} and are not covered by Lemma~\ref{smallds}. Checking using the database of DSs~\cite{Gordon} implies that there is no other such a DS with $v\leq 100000$. We are not able to find RDSs with parameters satisfying one of the conditions $(n,\lambda)=(2,2)$, $(n,\lambda)=(4,2)$, $\lambda=2$ which can be used for the construction from Proposition~\ref{rds} and are not covered by Lemma~\ref{smallrds}.

We would like to mention that it is possible to consider several other Cayley graphs over generalized dihedral groups arising from DSs and RDSs and not covered by Propositions~\ref{ds} and~\ref{rds}. One can obtain necessary and sufficient conditions for these graphs to be Deza graphs in terms of parameters of DSs or RDSs. We could find appropriate RDSs only for two infinite families among these graphs which are exactly divisible design graphs (see~\cite{HKM} for the definition) from~\cite[Corollary~6.3]{Ry}. Observe also that the graphs from Propositions~\ref{ds} and~\ref{rds} are not divisible design graphs.

\end{document}